\documentclass{amsart}
\usepackage{graphicx,xcolor,amssymb}

\usepackage{verbatim}

\usepackage{hyperref}

\usepackage[margin=3cm]{geometry}

\newcommand{\e}{\text{e}}
\newcommand{\R}{\mathbb{R}}
\newcommand{\Z}{\mathbb{Z}}
\renewcommand{\S}{\mathbb{S}}
\newcommand{\ds}{\,\mathrm ds}

\newcommand{\Kosc}{K_{\mathrm{osc}}}
\newcommand{\ra}[1]{#1}

\newcommand{\tGamma}{t^{\Gamma}}
\newcommand{\CGNS}{C_{G\!N\!S}}
\newcommand{\rot}{\text{rot }}

\renewcommand{\theequation}{\arabic{section}.\arabic{equation}}

\newtheorem{definition}{Definition}

\newtheorem{theorem}[definition]{Theorem}
\newtheorem{proposition}[definition]{Proposition}
\newtheorem{lemma}[definition]{Lemma}
\newtheorem{corollary}[definition]{Corollary}

\makeatletter
\title{On the planar free elastic flow with small oscillation of curvature}
\author{Ben Andrews${}^1$}\thanks{${}^1$: ben.andrews@anu.edu.au, ORCID 0000-0002-6507-0347, Mathematical Sciences Institute, Australian National University, Australia.}
\author{Glen Wheeler${}^2$}\thanks{${}^2$: glenw@uow.edu.au, ORCID 0000-0003-3314-5647, School of Mathematics and Physics, University of Wollongong, Australia}

\begin{document}

\begin{abstract}
The free elastic flow that begins at any  closed curve exists for all time.
If the initial curve is an $\omega$-fold covered circle (``$\omega$-circle'') the solution expands self-similarly.
Very recently, Miura and the second author showed that (topological) $\omega$-circles that are close to multiply-covered round circles are asymptotically stable under the planar free elastic flow, which means that upon rescaling the rescaled flow converges smoothly to the stationary (in the rescaled setting) $\omega$-circle.
Closeness in that work was measured via the derivative of the curvature scalar.
In the present paper, we improve this by requiring closeness in terms of the curvature scalar itself.
The convergence rate we obtain is sharp.
\end{abstract}

\maketitle

\section{Introduction}

Euler's elastic energy of a smooth closed immersed plane curve
$\Gamma:\mathbb S^1\to\mathbb R^2$ with arclength parameter $s$ (here $\S^1 = \R/L\Z$ where $L$ is the length of $\Gamma$) and curvature scalar $k$ is
\[
\mathcal E(\Gamma) = \int k^2\,\mathrm ds.
\]
Its $L^2(ds)$-gradient flow $\Gamma:\S^1\times[0,T)\to\R^2$ is the fourth-order evolution
\begin{equation}\label{eq:free-elastic}
\partial_t\Gamma = -\big(2k_{ss}+k^3\big)\nu = F(\Gamma(\cdot,t))\nu,
\quad\quad
F(\Gamma) := -(2k_{ss}+k^3),
\end{equation}
where $s$ subscripts denote derivatives with respect to the arclength parameter, 
and $\nu$ is the unit normal vector, defined here as the counter-clockwise rotation of $\tau = \partial_s\gamma$ through $\frac\pi2$. 
Along \eqref{eq:free-elastic} the energy is strictly decreasing:
\[
\frac{\mathrm d}{\mathrm dt}\,\mathcal E(\Gamma(\cdot,t))
 = -\int \big(2k_{ss}+k^3\big)^2\,\mathrm ds
 < 0,
\]
 as there are no closed free elastica.
The flow exists smoothly (each $\Gamma(\cdot,t)$ is smooth) for all time \cite{DKS02}. 
Thus we may always take $T=\infty$, and do so for the remainder of the article.

Circles (and, more generally, $\omega$-circles) $C_\rho$ with radius $\rho>0$ are special: they expand self-similarly under \eqref{eq:free-elastic}, with radius $\rho = \rho(t)$ satisfying 
\(
\rho(t) = \left(\rho_0^4 + 4t\right)^{\frac14}
\,.
\)
It is natural therefore to investigate the asymptotic stability of expanding solutions.
We factor out the expansion by passing to a continuous rescaling, and changing the time parameter.
This leads to the rescaled free elastic flow $\gamma:\S^1\times[0,\infty)\to\R^2$ satisfying 
\begin{equation}\label{eq:rescaled}
\partial_t\gamma  =  -\Big(2k_{ss}+k^3-\lambda(t)\,(\gamma\cdot\nu)\Big)\,\nu,
\quad
\lambda(t) = \frac1{2\omega\pi}\left(2 \int k_s^2\,\mathrm ds- \int k^4\,\mathrm ds\right),
\quad L(\gamma(\cdot,t)) = 2\omega\pi.
\end{equation}
On the rescaled flow we work always in the arclength parametrisation, and so $\S^1 = \R/2\omega\pi\Z$.

Let us briefly describe how to obtain \eqref{eq:rescaled} from \eqref{eq:free-elastic}.
Suppose we have a solution $\Gamma:\S^1\times[0,T)\to\R^2$ to \eqref{eq:free-elastic}.
Fix length to $2\omega\pi$ (so that the curvature of an $\omega$-circle is one) and define the scaling factor
\(
\sigma(\tGamma) := \frac{L_\Gamma(\tGamma)}{2\omega\pi}\,.
\)
Here and below we use the super/subscript $\Gamma$ to denote quantities on the unrescaled flow $\Gamma$. 
Rescale space and reparametrise time by
\begin{equation}\label{eq:rescaling-map}
\gamma(\cdot,t(\tGamma))\ :=\ \sigma(\tGamma)^{-1}\,\Gamma(\cdot,\tGamma),
\qquad\text{ with }\qquad \frac{dt}{d\tGamma} = \sigma(\tGamma)^{-4},\quad t(0)=0.
\end{equation}
Any tangential terms produced by this change are absorbed by a reparametrisation and do not
affect the geometry.

For the rescaled curve $\gamma$
\[
ds_\Gamma=\sigma\,ds,\qquad
k_\Gamma=\sigma^{-1}k,\qquad
(k_\Gamma)_s=\sigma^{-2}k_s,\qquad
(k_\Gamma)_{ss}=\sigma^{-3}k_{ss}.
\]
Consequently \(F(\Gamma(\cdot,\tGamma))\nu_\Gamma
= 
-\big(2(k_\Gamma)_{ss}+k_\Gamma^3\big)\nu_\Gamma
\ =\ -\sigma^{-3}\big(2k_{ss}+k^3\big)\nu.
\)
Differentiating \eqref{eq:rescaling-map} with respect to $\tGamma$ gives
\[
\partial_\tGamma\gamma
=  \partial_t\gamma\ \frac{dt}{d\tGamma}
= -\frac{\sigma'}{\sigma}\,\gamma
  +\sigma^{-1}\partial_\tGamma\Gamma
= -\frac{\sigma'}{\sigma}\,(\gamma\cdot\nu)\,\nu - \sigma^{-4}\big(2k_{ss}+k^3\big)\nu
\ +\ \text{(tangential)}.
\]
Simplifying yields
\[
\partial_t\gamma
= -\Big(2k_{ss}+k^3-\underbrace{\big(-\sigma^3\sigma'\big)}_{=:\,\lambda(t)}(\gamma\cdot\nu)\Big)\nu.
\]
Thus $\gamma$ satisfies the rescaled flow \eqref{eq:rescaled} as claimed.
To identify $\lambda$, note that
\[
\frac{d}{d\tGamma}L_\Gamma(\Gamma(\cdot,\tGamma))
= - \int_{\Gamma} k_\Gamma\,\big(-2(k_\Gamma)_{ss}-k_\Gamma^3\big)\,ds_\Gamma
= -2 \int_{\Gamma}(k_\Gamma)_s^2\,ds_\Gamma+ \int_{\Gamma}k_\Gamma^4\,ds_\Gamma,
\]
and hence (using the scaling rules above)
\[
\sigma'(\tGamma)=\frac{1}{2\omega\pi}\frac{d}{d\tGamma}L_\Gamma
= \frac{1}{{2\omega\pi}\sigma^3} \left(-2 \int k_s^2\,ds+ \int k^4\,ds\right).
\]
This implies $\lambda$ is given by the second expression in \eqref{eq:rescaled}.

Using the identity $\int k\,(\gamma\cdot\nu)\,ds=-{2\omega\pi}$ (which follows using $k\nu = \gamma_{ss}$ and integration by parts) one checks
immediately that
\[
\frac{d}{dt}\,L(\gamma(\cdot,t))
= - \int k\Big(-2k_{ss}-k^3+\lambda(\gamma\cdot\nu)\Big)\,ds
= \Big(-2 \int k_s^2+ \int k^4\Big) - \lambda \int k(\gamma\cdot\nu)\,ds
=0,
\]
so the rescaled evolution preserves length and the unit $\omega$-circle $\gamma_\omega:\S^1\to\R^2$ centred at the origin is stationary for \eqref{eq:rescaled}.


The long-time dynamics of \eqref{eq:rescaled} near $\omega$-circles is the focus of this paper. Miura and the second author \cite{MW25} proved that  the flow with initial data sufficiently close to an $\omega$-circle (measured at the level of $\partial_s k$) is asymptotically stable: the rescaled flow converges smoothly to an $\omega$-circle. The present work improves this by removing one derivative from the smallness hypothesis.
A natural, scale-invariant way to quantify curvature-level closeness is via the  {normalised oscillation of curvature}
\begin{equation}\label{eq:Kosc}
K_{\mathrm{osc}}(\gamma(\cdot,t)) := L(\gamma(\cdot,t))\,\big\|k-\bar k\big\|_{L^2(ds)}^2,
\qquad
\bar k = \frac1{L(\gamma(\cdot,t))}\int  k\,\mathrm ds  = \frac{2\omega\pi}{L(\gamma(\cdot,t))},
\end{equation}
where $\omega\in\mathbb N$ is the turning number of $\gamma$. The quantity $K_{\mathrm{osc}}$ vanishes precisely for $\omega$-circles and is invariant under rescaling.
Our main result is the following.


\begin{theorem}\label{thm:main}
For each turning number $\omega>0$ there exist $\varepsilon_{\omega},\delta_\omega>0$ with the following property.
Let $\gamma_0$ be a smooth closed immersed plane curve with turning number $\omega$, length equal to $2\omega\pi$, and
\begin{equation}
\label{eq:koscsmall}
\Kosc(\gamma_0) \le \varepsilon_{\omega}.
\end{equation}
The rescaled free elastic flow \eqref{eq:rescaled} with initial data $\gamma_0$
exists for all time and converges to the stationary $\omega$-circle $\gamma_\omega$ centred at the origin, with
\[
\Kosc(\gamma(\cdot,t)) \le C_\omega\,e^{-\left(\frac74+\delta_\omega\right) t}\,,\qquad t\ge0\,.
\]
Moreover, for a constant $C=C(\gamma_0,\omega)$,
\begin{equation}
\label{eq:d2conv}
  \mathbf d_2(\gamma(\cdot,t), \gamma_\omega)
\ \le\ C\,e^{-\frac{1}{2}\left(\frac74+r_\omega\right)t}\,,
\end{equation}
where $r_\omega:=\min\{\frac14,\delta_\omega\}\in(0,1/4]$, and $\mathbf d_2$ is the invariant $W^{2,2}$ distance (see Definition \ref{def:Sob-GH}).
For $\omega\in\{1,2,3,4\}$, $r_\omega = 1/4$, and in general $r_\omega = O(\omega^{-2})$.
\end{theorem}

The rate of convergence of the curvature to 1 is sharp and depends dramatically on $\omega$.
The precise definition of $\delta_\omega$ is given in \eqref{eq:defndelta2}; here, we collect key observations about the value of $\delta_\omega$ for various $\omega$.
Asymptotically, $\delta_\omega = O(\omega^{-2})$.
While $\delta_\omega$ always controls the rate of decay of the curvature to 1, the rate of convergence for the parametrisation $\gamma$ to $\gamma_\omega$ involves the convergence rate of translations as well, as seen by the use of $r_\omega$ in \eqref{eq:d2conv}.
The critical value for $\delta_\omega$ is $1/4$: if $\delta_\omega>1/4$, then the rate of convergence for the parametrisation is determined by translations, and if $\delta_\omega<1/4$ the rate is determined by curvature. 
 For $\omega$ small, $\delta_\omega$ is much larger than $1/4$, dipping below this critical value for $\omega\ge5$.
In particular, $\delta_1 = 121/4$, $\delta_2 = 4$, $\delta_3 = 361/324$, $\delta_4 = 25/64$, and $\delta_5 = 361/2500 < 1/4$.
 Thus, for $\omega\in\{1,2,3,4\}$, the convergence rate is limited by that of the translation, and for $\omega\ge5$ the convergence rate is limited by that of curvature.
 
It is possible to remove the influence of translations entirely by making a pre-eminent choice of origin for the unrescaled flow.
 
\begin{theorem}\label{thm:preeminent-origin}
In the setting of Theorem~\ref{thm:main},  the centre of mass of the unrescaled flow converges to a point
\[
p_\infty
:=\lim_{\tGamma\to\infty}\frac1{L_\Gamma(\tGamma)}\int \Gamma\,ds_\Gamma
\ \in\ \R^2.
\]
Define the translated flow and its length-normalisation by
\[
\Gamma^\sharp(\cdot,\tGamma):=\Gamma(\cdot,\tGamma)-p_\infty,
\qquad
\gamma^\sharp(\cdot,t(\tGamma)):=\sigma(\tGamma)^{-1}\Gamma^\sharp(\cdot,\tGamma).
\]
Then there exist constants $C=C(\omega,\Gamma_0)>0$ and $\tGamma_0=\tGamma_0(\omega,\Gamma_0)\ge0$ (where $\Gamma_0 = \Gamma(\cdot,0)$) such that
\[
\mathbf d_2\big(\gamma^\sharp(\cdot,t(\tGamma)),\gamma_\omega\big)
\ \le\ C\,(1+\tGamma)^{-\,\frac18\left(\frac74+\delta_\omega\right)}
\qquad(\tGamma\ge\tGamma_0).
\]
Equivalently, 
\[
\mathbf d_2\big(\gamma^\sharp(\cdot,t),\gamma_\omega\big)
\ \le\ C\,\mathrm e^{-\,\frac12\left(\frac74+\delta_\omega\right)t}
\qquad(t\ge t(\tGamma_0)).
\]
In particular, after this fixed translation of the unrescaled flow, the $\mathbf d_2$ convergence rate
is governed solely by the curvature decay exponent $\frac74+\delta_\omega$ and is not limited by the translation mode.
\end{theorem}

The key parts of our proof are two new integral estimates, that enable us to show that the hypothesis of \cite[Theorem 1.2]{MW25} is eventually satisfied.
Beyond this, we use some novel methods to control translation and deduce decay of the position vector from decay of the curvature.
The first integral estimate provides preservation of the smallness condition \eqref{eq:koscsmall} (and  its exponential decay), and the second integral estimate gives eventual smallness of $||k_s||_2^2$ under the condition \eqref{eq:koscsmall}.
In each case we use elementary arguments.
The first estimate uses a `linearisation'-type method, which is facilitated by a spectral gap phenomenon that enables us to obtain the sharp rate of convergence.
For the second estimate, the decisive point is a favourable cancellation among
the zero-order terms.

\section*{Acknowledgments}
This work was completed while the authors were participating in the Research Program ``Gradient Flows in Geometry and PDE'' (January 2025) at  the MATRIX institute in Creswick, Victoria, Australia.  The authors are grateful to this institute for hosting a valuable meeting that included generous collaboration time.
The first author was additionally supported by ARC grants FL150100126 and DP250103952, and the second author by ARC grants DP250101080 and FT250100880.

\section{Preservation and improvement of \texorpdfstring{$\Kosc$}{Kosc}}

We start by determining the evolution of  the rescaled oscillation of curvature:
\[
e(t):=\int (k-1)^2\,\mathrm ds  = \int k^2\,\mathrm ds-2 \int k\,\mathrm ds+2\omega\pi
 = \int k^2\,\mathrm ds-2\omega\pi\,.
\]
We also use the notation $f:= k-1$ below (and then $e(t) = ||f||_2^2$).

\begin{lemma}\label{lem:ede}
Let $\gamma:\S^1\times[0,\infty)\to\R^2$ be a rescaled free elastic flow \eqref{eq:rescaled}.
Then
\begin{equation}\label{eq:eprime-exact}
\frac{\mathrm d}{\mathrm dt}e(t)
 = -\int \big(2k_{ss}+k^3\big)^2\,\mathrm ds  - \lambda(t)\int k^2\,\mathrm ds.
\end{equation}
\end{lemma}
\begin{proof}
As usual, we have the evolution equations $k_t=(F+\lambda(\gamma\cdot\nu))_{ss}+k^2(F+\lambda(\gamma\cdot\nu))$ and $(\mathrm ds)_t=-k(F+\lambda(\gamma\cdot\nu))\,\mathrm ds$. Hence
\begin{align*}
\frac{\mathrm d}{\mathrm dt}\int k^2\,\mathrm ds
&=\int \big(2k\,k_t-k^3(F+\lambda(\gamma\cdot\nu))\big)\,\mathrm ds
\\&=\int \big(2k(F+\lambda(\gamma\cdot\nu))_{ss}+k^3(F+\lambda(\gamma\cdot\nu))\big)\,\mathrm ds
\\&=\int \big(2k_{ss}+k^3\big)(F+\lambda(\gamma\cdot\nu))\,\mathrm ds.
\end{align*}
We now claim
\begin{equation}\label{eq:key-identity}
\int (2k_{ss}+k^3)\,(\gamma\cdot\nu)\,\mathrm ds  =  -\int k^2\,\mathrm ds.
\end{equation}
It follows by two integrations by parts using
$(\gamma\cdot\nu)_s=-k(\gamma\cdot\tau)$ and $(\gamma\cdot\tau)_s=1+k(\gamma\cdot\nu)$.
Since $\int k\,\mathrm ds=2\omega\pi$ is topological and $L(\gamma(\cdot,t))$ is preserved by \eqref{eq:rescaled}, we have $\frac{\mathrm d}{\mathrm dt}e=\frac{\mathrm d}{\mathrm dt}\int k^2\,\mathrm ds$. Combining the previous equations yields \eqref{eq:eprime-exact}.
\end{proof}

\begin{proposition}\label{prop:quad}
Let $\gamma:\S^1\times[0,\infty)\to\R^2$ be a rescaled free elastic flow \eqref{eq:rescaled}.
Assume $e(t)\le 1$. Then  
\begin{align}
\frac{\mathrm d}{\mathrm dt}e
&= \ra{-Q(f)} + \mathcal R[f], \label{eq:lin-line}
\end{align}
\ra{where
\[
Q(f):=4\int f_{ss}^2\,\mathrm ds
 -10\int f_s^2\,\mathrm ds
 +8\int f^2\,\mathrm ds .
\]}
and where the remainder satisfies the bound
\begin{equation}\label{eq:Rbound}
\mathcal R[f] \le  C\,e^{ 1/2}\,\int f_{ss}^2\,\mathrm ds  +  C\,e^{ 3/2},
\end{equation}
for a universal constant $C>0$ independent of $\omega$. In particular,
\begin{equation}\label{eq:ede-quad}
\frac{\mathrm d}{\mathrm dt}e  \le  -\Big(4-Ce^{ 1/2}\Big) \int f_{ss}^2\ds
 + 10\int f_s^2\ds
 - \Big(8-Ce^{ 1/2}\Big) \int f^2\ds
.
\end{equation}
\end{proposition}
\begin{proof}
\ra{We use the sign convention
\begin{equation}
\label{EQlinearpart}
Q(f) := 4\int f_{ss}^2\ds
- 10\int f_s^2\ds
+ 8\int f^2\ds.    
\end{equation}}
Our approach is to first collect the quadratic (in $f$ and its derivatives) terms to form \ra{$-Q(f)$}, which we shall highlight by boxing them, then label all remaining terms (cubic and higher) as  $\mathbf{(R1)}-\mathbf{(R12)}$. 

\medskip\noindent
\textit{1) The square $\boldsymbol{-\int(2k_{ss}+k^3)^2\ds}$.}
Using $k_{ss}=f_{ss}$ and $(1+f)^3=1+3f+3f^2+f^3$,
\begin{align*}
-\int \big(2k_{ss}+k^3\big)^2\ds
&= -\int \Big(4f_{ss}^2 + 4f_{ss}(1+3f+3f^2+f^3) + (1+f)^6\Big)\ds
\\
&= \boxed{-4 \int f_{ss}^2\ds}
      + \underbrace{\mathbf{(B)} }_{\text{cross term}}
      + \underbrace{\mathbf{(C)}}_{\text{pure powers}},
\end{align*}
where
\[
\mathbf{(B)}=-4 \int f_{ss} \ds 
-12 \int f f_{ss}\ds
-12 \int f^2 f_{ss}\ds
-4 \int f^3 f_{ss}\ds,
\qquad
\mathbf{(C)}=-\int (1+f)^6\ds.
\]

\emph{Cross term $\mathbf{(B)}$.}
Periodicity and integration by parts imply
\[
-4 \int f_{ss}=0,\qquad
-12 \int f f_{ss} = \boxed{+12 \int f_s^2\ds}.
\]
The second term above is quadratic and will be collected later to form \ra{$-Q$}. 
The remaining two pieces of $\mathbf{(B)}$ are kept as remainder terms:
\[
{\mathbf{(R1)} := -12 \int f^2 f_{ss}\ds},\qquad
{\mathbf{(R2)} := -4 \int f^3 f_{ss}\ds}.
\]
\emph{Pure powers $\mathbf{(C)}$.} Using $(1+f)^6=1+6f+15f^2+20f^3+15f^4+6f^5+f^6$ and $\int f\ds=0$:
\begin{align*}
\mathbf{(C)}
&= - 2\omega\pi
    + \boxed{-15e}
      + \mathbf{(R3)}
      + \mathbf{(R4)}
      + \mathbf{(R5)}
      + \mathbf{(R6)}
\end{align*}
where
\[
{\mathbf{(R3)} := -20 \int f^3\ds},\quad
{\mathbf{(R4)} := -15 \int f^4\ds},\quad
{\mathbf{(R5)} := -6 \int f^5\ds},\quad
{\mathbf{(R6)} := - \int f^6\ds}.
\]

\medskip\noindent
\textit{2) The rescaling term $\boldsymbol{-\lambda\int k^2\ds}$.}
We calculate (using \eqref{eq:rescaled})
\begin{align*}
-\lambda \int k^2\ds
&= -\frac1{2\omega\pi} \left(2 \int f_s^2\ds
- \int (1+f)^4\ds\right)({2\omega\pi}+e) \\
&= \underbrace{-2 \int f_s^2\ds
+ \int (1+f)^4\ds}_{\text{free part}}
     +  \underbrace{\Big( -\frac{2}{{2\omega\pi}}\,e \int f_s^2\ds
   + \frac{1}{{2\omega\pi}}\,e \int (1+f)^4\ds\Big)}_{\text{$e$-coupled part}}.
\end{align*}
\emph{Free part.} Expanding, 
\[
-2 \int f_s^2\ds
+ \int (1+f)^4\ds
= \boxed{-2 \int f_s^2\ds}  +  {2\omega\pi}  +  \boxed{6e} + \mathbf{(R7)} + \mathbf{(R8)}
\]
where
\[
{\mathbf{(R7)} := 4 \int f^3\ds},\quad
{\mathbf{(R8)} := \int f^4\ds}.
\]
\emph{$e$-coupled part.} Continuing, 
\begin{align*}
-&\frac{2}{{2\omega\pi}}e \int f_s^2\ds + \frac{1}{{2\omega\pi}}e \int (1+f)^4\ds
= 
{\mathbf{(R9)}}
   + \boxed{ e }
   +  
  {\mathbf{(R10)}}
   +  {\mathbf{(R11)}}
   +  {\mathbf{(R12)}}
\end{align*}
where
\[
{\mathbf{(R9)} := -\frac{2}{{2\omega\pi}}\,e \int f_s^2\ds}
,\quad
{\mathbf{(R10)} := \frac{6}{{2\omega\pi}}\,e^2}
,\quad
{\mathbf{(R11)} := \frac{4}{{2\omega\pi}}\,e \int f^3\ds}
,\quad
{\mathbf{(R12)} := \frac{1}{{2\omega\pi}}\,e \int f^4\ds}.
\]

\medskip\noindent
\textbf{3) Collecting the quadratic  part.}
Adding the boxed terms:
\begin{align*}
\boxed{-4 \int f_{ss}^2\ds}
 &+ \boxed{(12 - 2) \int f_s^2\ds}
 + \boxed{(-15 + 6 + 1) \int f^2\ds}
\\ &= 
-4 \int f_{ss}^2\ds  + 10 \int f_s^2\ds  - 8 \int f^2\ds \ra{= -Q(f)}.
\end{align*}
All other terms are precisely $\mathbf{(R1)}$-$\mathbf{(R12)}$ as boxed above.
To estimate them, we use the Gagliardo-Nirenberg Sobolev inequality \cite{GNS}, stated here for $n=1$ and periodic functions with zero average:
\begin{equation}
\label{eq:gns}
||f_{s^j}||_{L^p} \le \CGNS||f||_{L^q}^{1-\theta}||f_{s^m}||_{L^r}^\theta,
\qquad
\frac1p = j + \theta\left(\frac1r-m\right)+\frac{1-\theta}{q}
\end{equation}
where (for instance) $j/m \le \theta \le 1$, $p\in[1,\infty)$, and $q,r\in[1,\infty]$.
Note that $\CGNS=\CGNS(p,j,\theta,r,m,q)$ does not depend on $f$.
A particularly useful consequence of \eqref{eq:gns} that we use here is when  $j=0$, $m=r=q=2$, which we record below:
%
%
%
%
\begin{equation}
\label{eq:gns-cons}
||f||_{L^p} 
\le \CGNS||f||_{L^2}^{\frac34 + \frac1{2p}}
    ||f_{ss}||_{L^2}^{\frac14 - \frac1{2p}}
    = \CGNS\, e^{\frac38 + \frac1{4p}}
    ||f_{ss}||_{L^2}^{\frac14 - \frac1{2p}}
    ,
\qquad p\ge2\,.
\end{equation}
As $\CGNS$ (for each choice of the parameters $p,j,\theta,r,m,q$) is a universal constant, we absorb it into the $C$ in the estimates below (a universal constant that may change from line to line). We shall explicitly track  $\omega$, which $C$ does not depend on.

Let us now estimate each term in turn.

\medskip
\noindent\textbf{(R1)} $\displaystyle = -12\int f^2 f_{ss}\ds$.  
 Use the H\"older inequality, \eqref{eq:gns-cons} with $p=4$, then Young:
\begin{align*}
|\mathbf{(R1)}| 
 \le 12\|f\|_{L^4}^2\|f_{ss}\|_{L^2}
\le C\,e^{\frac78}\|f_{ss}\|_{L^2}^{\frac14}
\|f_{ss}\|_{L^2}
\le
\,e^{1/2}\,\|f_{ss}\|_{L^2}^2 + C\,e^{3/2}\,.
\end{align*}

\medskip
\noindent\textbf{(R2)} $\displaystyle = -4\int f^3 f_{ss}\ds$.
We apply Young's inequality, \eqref{eq:gns-cons} with $p=6$, then Young again and finally $e\le1$:
\begin{align*}
|\mathbf{(R2)}|
\le \frac12 e^{\frac12}\|f_{ss}\|_{L^2}^2
 + 8 e^{-\frac12}\|f\|_{L^6}^6
\le \frac12 e^{\frac12}\|f_{ss}\|_{L^2}^2
 + C e^{-\frac12 + \frac52}\|f_{ss}\|_{L^2}
\le \,e^{\frac12}\,\|f_{ss}\|_{L^2}^{2}
 + C e^{\frac{3}{2}}
 .
\end{align*}

\medskip
\noindent\textbf{(R3)+(R7)} $\displaystyle = -16\int f^3\ds$.  
 We use \eqref{eq:gns-cons} with $p=3$ and then Young to find
 %

 %
\begin{align*}
\Big|(R3)+(R7)\Big|
&\le C\,e^{\frac{11}{8}}\|f_{ss}\|_{L^2}^{\frac14}
\le e^{\frac12}\,\|f_{ss}\|_{L^2}^2 + C\, e^{\frac32}
\,.
\end{align*}

\medskip
\noindent\textbf{(R5)}+(\textbf{(R4)}+\textbf{(R8)})+\textbf{(R6)} $\displaystyle = -6\int f^5\ds
-14\int f^4\ds
-\int f^6\ds$.
Since $-6f^5 \le f^6 + 9f^4$, 
\begin{align*}
-6\int f^5\ds
-14\int f^4\ds
-\int f^6\ds
\le (1-1)\int f^6\ds
    + (9-14)\int f^4\ds
\le -5\int f^4\ds
\le 0\,.
\end{align*}
We do not need to estimate the absolute value of the remainder terms.\footnote{It would be possible to estimate the absolute value of the remainder terms, but it would take additional space, and we do not see any benefit in pursuing this here. Similarly one may keep a number of `good' terms with favourable signs.}.

\medskip
\noindent\textbf{(R9)} $\displaystyle =-\frac{1}{\omega\pi}\,e\int f_s^2\ds \le 0$.  

\medskip
\noindent\textbf{(R10)} $\displaystyle =\frac{3}{\omega\pi}\,e^2 \le C\omega^{-1} e^{\frac32}$ (using $e\le1$).

\medskip
\noindent\textbf{(R11)} $\displaystyle =\frac{2}{\omega\pi}\,e\int f^3\ds$.  
Using $e\le1$ and the same method as for \textbf{(R3)}+\textbf{(R7)} (note the $\omega$):
\begin{align*}
\Big|\textbf{(R11)}\Big|
&\le C\omega^{-1}\,e^{\frac{19}{8}}\|f_{ss}\|_{L^2}^{\frac14}
\le e^{\frac12}\,\|f_{ss}\|_{L^2}^2 + C\omega^{-\frac87}\, e^{\frac32}
\,.
\end{align*}

\medskip
\noindent\textbf{(R12)} $\displaystyle =\frac{1}{2\omega\pi}\,e\int f^4\ds$.  
 Using \eqref{eq:gns-cons} with $p=4$, then Young, then $e\le1$, we find
\begin{align*}
|\mathbf{(R12)}|
&\le C\omega^{-1}\,e^{\frac{11}{4}}\|f_{ss}\|_{L^2}^{\frac12}
\le e^{\frac12}\,\|f_{ss}\|_{L^2}^2
 + C\omega^{-\frac43} e^{\frac{7}2}
\le e^{\frac12}\,\|f_{ss}\|_{L^2}^2
 + C\omega^{-\frac43} e^{\frac32}
\end{align*}
since $e\le 1$.

\medskip
\noindent\emph{Conclusion.}
Collecting $\mathbf{(R1)}$-$\mathbf{(R12)}$ we obtain 
\[
\mathcal R[f]
 \le  C\,e^{ 1/2}\,\|f_{ss}\|_{L^2}^2 
 + C\left(
    1 + \omega^{-\frac43} + \omega^{-\frac87} + \omega^{-1}
 \right)\,e^{3/2}
 \le  C\,e^{ 1/2}\,\|f_{ss}\|_{L^2}^2 
 + C\,e^{3/2},
\]
for a universal constant $C$ (note that $\omega\ge1$), which is precisely \eqref{eq:Rbound}. 
Combining \eqref{eq:Rbound} with \eqref{eq:lin-line} yields \eqref{eq:ede-quad} and completes the proof.
\end{proof}

\begin{lemma}\label{lem:gap}
For every $\omega\in\mathbb N$ and every $f\in H^2$ with zero mean one has
\begin{equation}\label{eq:coerciveQ}
4\int f_{ss}^2\,\mathrm ds-10\int f_s^2\,\mathrm ds+8\int f^2\,\mathrm ds
 \ge \lambda_\omega\,\int f^2\,\mathrm ds,
\end{equation}
with
\begin{equation}
\label{eq:defndelta}
\lambda_\omega = \min_{n\in\mathbb Z\setminus\{0\}}\Big\{4\left(\frac{n}{\omega}\right)^4-10\left(\frac{n}{\omega}\right)^2+8\Big\}
\,.
\end{equation}
\end{lemma}
\begin{proof}
Expand $f$ in the Fourier basis on the circle of length $2\omega\pi$ (note $a_0=0$ as $f$ has zero mean):
{$f(s)=\sum_{n\in\mathbb Z\setminus\{0\}}a_n \e^{i n s/\omega}$.}
{Parseval gives}
\[
\int f^2\,\mathrm ds=2\omega\pi\sum_{n\in\mathbb Z\setminus\{0\}}|a_n|^2,\ \ 
\int f_s^2\,\mathrm ds=2\omega\pi\sum_{n\in\mathbb Z\setminus\{0\}}\Big(\frac{n}{\omega}\Big)^2|a_n|^2,\ \ 
\int f_{ss}^2\,\mathrm ds=2\omega\pi\sum_{n\in\mathbb Z\setminus\{0\}}\Big(\frac{n}{\omega}\Big)^4|a_n|^2.
\]
So
\begin{align*}
4\int f_{ss}^2\,\mathrm ds-10\int f_s^2\,\mathrm ds+8\int f^2\,\mathrm ds
&=
\sum_{n\in\mathbb Z\setminus\{0\}} \left(
4\left(\frac{n}{\omega}\right)^4-10\left(\frac{n}{\omega}\right)^2+8\right) 2\omega\pi|a_n|^2
\\&\ge \lambda_\omega\, 2\omega\pi\sum_{n\in\mathbb Z\setminus\{0\}}|a_n|^2
= \lambda_\omega\,\int f^2\,\mathrm ds,
\end{align*}
as required.
\end{proof}

When we take $f = k-1$, we may use a novel estimate to further improve the spectral gap (and eventually obtain a better convergence rate).
 
\begin{corollary}\label{cor:gap-kminus1}
Let $\gamma:\S^1\to\R^2$ be a smooth immersed curve with length $2\omega\pi$ and set $f = k-1$ where $k$ is the curvature of $\gamma$.
Then
\begin{equation}\label{eq:gap-kminus1}
4\int f_{ss}^2\,\mathrm ds-10\int f_s^2\,\mathrm ds+8\int f^2\,\mathrm ds
 \ge \widehat\lambda_\omega\,\int f^2\,\mathrm ds
 - 4\omega\pi\widehat\lambda_\omega\left(\int f^2\,\mathrm ds\right)^{ 2},
\end{equation}
where 
\begin{equation}
\label{eq:defndelta2}
\widehat\lambda_\omega
:=\min_{n\in\mathbb Z\setminus\{0,\pm\omega\}}
\Big\{\,4\Big(\tfrac{n}{\omega}\Big)^4
       -10\Big(\tfrac{n}{\omega}\Big)^2
       +8\,\Big\}\,.
\end{equation}
\end{corollary}
\begin{proof}
Consider the first harmonics of $f$:
\[
I_c:=\int f(s)\cos s\,\mathrm ds,\qquad
I_s:=\int f(s)\sin s\,\mathrm ds.
\]
We claim that
\begin{equation}\label{eq:nu-cos-est}
\sup_{s\in[0,2\omega\pi]} \big|\nu(s)-(\cos s,\sin s)\big|
 \le  \int |k-1|\,\mathrm ds.
\end{equation}
To see this, let $R_\theta$ denote counter-clockwise rotation (and $\rot = R_{\frac\pi2}$) by angle $\theta\in\R$, and calculate
\begin{align*}
(R_{-s}\nu)_s
 = -\rot R_{-s}\nu - k R_{-s}\tau
 = -R_{-s}(k-1)\tau
\end{align*}
so $|(R_{-s}\nu)_s| = |k-1|$.
Rotating $\gamma$ so that $\nu(0) = (1,0)$ we find, using that $R_s$ is an isometry and the fundamental theorem of calculus, that
\[
|\nu - (\cos s,\sin s)|
= |R_{-s}\nu - (1,0)|
\le \int |k(\sigma)-1|\,d\sigma,
\qquad s\in[0,2\omega\pi]\,,
\]
which implies \eqref{eq:nu-cos-est}.

\ra{Now, since
\(
\int_\S (k-1)\,\nu\,\mathrm ds =0,
\)
we estimate the first harmonics as a vector.  With
\(I:=(I_c,I_s)\),
\begin{align*}
|I|
&= \left|\int (k-1)(\cos s,\sin s)\,\mathrm ds\right|  
 = \left|\int (k-1)\big((\cos s,\sin s)-\nu(s)\big)\,\mathrm ds\right| \\
&\le \Big(\int |k-1|\,\mathrm ds\Big)
     \sup_s |\nu(s)-(\cos s,\sin s)|.
\end{align*}
Combining with \eqref{eq:nu-cos-est} yields
\begin{equation}\label{eq:first-harmonic-est}
|I|^2=|I_c|^2+|I_s|^2
 \le  \Big(\int |k-1|\,\mathrm ds\Big)^4
 \le  \left(2\omega\pi\int (k-1)^2\,\mathrm ds\right)^2
 =  (2\omega\pi\,e)^2.
\end{equation}}
Recall that
\(
\cos s=\tfrac12(\e^{i s}+\e^{-i s})
=\tfrac12(\e^{i\omega s/\omega}
              +\e^{-i\omega s/\omega})
\),
and similarly for $\sin s$, so orthogonality gives
\[
\int f(s)\cos s\,\mathrm ds
=\omega\pi\bigl(a_\omega+a_{-\omega}\bigr),
\qquad
\int f(s)\sin s\,\mathrm ds
=-i\omega\pi\bigl(a_{-\omega}-a_\omega\bigr).
\]
Solving for $a_\omega,a_{-\omega}$ we find $a_\omega = \frac1{2\omega\pi}(I_c-iI_s)$ and 
$a_{-\omega} = \frac1{2\omega\pi}(I_c+iI_s)$.
Using the identity
\(
|\tfrac{A+i B}{2}|^2
+|\tfrac{A-i B}{2}|^2
=\tfrac12(|A|^2+|B|^2),
\)
we find
\begin{equation}\label{eq:a-omega-from-IcIs}
|a_\omega|^2+|a_{-\omega}|^2
=\frac{1}{2\omega^2\pi^2}\Big(|I_c|^2+|I_s|^2\Big).
\end{equation}

Now decompose
\[
e
{=2\omega\pi\sum_{n\in\mathbb Z\setminus\{0\}}|a_n|^2} 
{= 2\omega\pi\sum_{n\in\mathbb Z\setminus\{0,\pm\omega\}}|a_n|^2
+ 2\omega\pi\big(|a_\omega|^2+|a_{-\omega}|^2\big)}
= E_{\mathrm{rest}} + E_{\mathrm{tr}}
\]
where
\[
E_{\mathrm{rest}} := 2\omega\pi\sum_{n\in\mathbb Z\setminus\{0,\pm\omega\}}|a_n|^2
\]
and, combining \eqref{eq:first-harmonic-est} and \eqref{eq:a-omega-from-IcIs}, 
\begin{equation}
E_{\mathrm{tr}} := 2\omega\pi\big(|a_\omega|^2+|a_{-\omega}|^2\big)
\le 4\omega\pi e^2.
\label{eq:a-omega-small}
\end{equation}
Define $p(x) = 4x^4 - 10x^2+8$; then
\begin{align*}
Q(f)
&= {2\omega\pi\sum_{n\in\mathbb Z\setminus\{0\}}p\Big(\frac{n}{\omega}\Big)|a_n|^2}
\ge \widehat\lambda_\omega\,E_{\mathrm{rest}}
 =  \widehat\lambda_\omega\,(e - E_{\mathrm{tr}})
 \ge  \widehat\lambda_\omega\,e - 4\omega\pi \widehat\lambda_\omega e^2,
\end{align*}
using \eqref{eq:a-omega-small} in the last step. This is exactly
\eqref{eq:gap-kminus1}.
\end{proof}

We will use the following ODE comparison lemma.

\begin{lemma}
\label{lem:ODE-comparison}
Let \(e:[0,\infty)\to[0,\infty)\) be \(C^1\) and suppose that for some \(\alpha>0\), \(C>0\), and \(\sigma>0\),
\begin{equation}\label{eq:gen-ode}
\frac{d}{dt}e(t)\ \le\ -\alpha\,e(t)\ +\ C\,e(t)^{\,1+\sigma}\qquad\text{whenever }e(t)\le 1.
\end{equation}
If
\[
e(0)\ \le\ \varepsilon_0\ :=\ \min\Big\{\frac12,\ \Big(\tfrac{\alpha}{C}\,(1-2^{-\sigma})\Big)^{ 1/\sigma}\Big\}
\]
we have
\begin{equation}\label{eq:gen-2exp}
e(t)\ \le\ 2\,e(0)\,e^{-\alpha t}\qquad\text{for all }t\ge0.
\end{equation}
\end{lemma}
\begin{proof}
If \(e(0)=0\), the conclusion is trivial. Assume \(e(0)>0\). 
Set 
\(
T_0:=\inf\{t>0:e(t)=1\}
\) (with the convention \(T_0=\infty\) if the set is empty); since \(e(0)<1\),  continuity implies \(T_0>0\). 
On every interval \([0,T]\) with \(T<T_0\), we have
\(e(t)<1\) (so \eqref{eq:gen-ode} is valid).

Define
\[
\bar e(t)
:=
e(0) e^{-\alpha t}
\left[
1-\frac{C}{\alpha}e(0)^\sigma
\bigl(1-e^{-\sigma\alpha t}\bigr)
\right]^{-1/\sigma}.
\]
Observe that $\bar e(0)=e(0)$ and
\[
\frac{d}{dt}\bar e(t)=-\alpha \bar e(t)+C\bar e(t)^{1+\sigma}
=f(\bar e(t)),
\qquad
\text{where}\qquad 
f(s):=-\alpha s+Cs^{1+\sigma}
.
\]
The smallness assumption gives
\(
\frac{C}{\alpha}e(0)^\sigma\le 1-2^{-\sigma}
\),
and hence, for every \(t\ge0\),
\(
1-\frac{C}{\alpha}e(0)^\sigma
\bigl(1-e^{-\sigma\alpha t}\bigr)
\ge
1-\frac{C}{\alpha}e(0)^\sigma
\ge
2^{-\sigma}\). 
Thus \(\bar e\) is globally well-defined and
\(
\bar e(t)\le 2e(0)e^{-\alpha t}.
\)

ODE comparison with \(e\) and \(\bar e\) on \([0,T]\) thus implies
\(e(t)\le \bar e(t) \le 2e(0)e^{-\alpha t}\) for
$0\le t\le T$.
As \(T<T_0\) was arbitrary, this holds for all \(0\le t<T_0\). 
We now argue that \(T_0=\infty\). 
If \(T_0<\infty\), then $e(T_0)=1$. 
By continuity and the above estimate,
\(
e(T_0)\le 2e(0)e^{-\alpha T_0}<1
\), a contradiction.
Hence \(T_0=\infty\), and \eqref{eq:gen-2exp} follows.
\end{proof} 

\begin{corollary}
\label{cor:ODE}
There exists $\varepsilon_\omega>0$ depending only on $\omega$ such that if
$e(0)\le\varepsilon_\omega$, then along the rescaled flow \eqref{eq:rescaled}
\begin{equation}\label{eq:ode-final-improved}
e(t) \le 2\,e(0)\,\mathrm e^{-\widehat\lambda_\omega t}
\qquad\text{for all }t\ge0,
\end{equation}
where
$\widehat\lambda_\omega=\frac74+\delta_\omega$ and $\delta_\omega>0$.
\end{corollary}
\begin{proof}
Expanding $f(s)=\sum_{n\in\mathbb Z}a_n \e^{i n s/\omega}$ and using the elementary inequality
\[
{x^4+1\le 2\big(4x^4-10x^2+8\big)\qquad\text{for all }x\in\mathbb R,}
\]
we obtain
\[
\int f_{ss}^2\,\mathrm ds+\int f^2\,\mathrm ds
{=2\omega\pi\sum_{n\in\mathbb Z}\left[\Big(\frac{n}{\omega}\Big)^4+1\right]|a_n|^2}
 \le  2Q[f].
\]
Hence, after enlarging $C$ if necessary,
\[
(4-Ce^{1/2}) \int f_{ss}^2 -10 \int f_s^2 +(8-Ce^{1/2}) \int f^2
 \ge (1-Ce^{1/2})\,Q[f],
\]
so, using Proposition~\ref{prop:quad}, we can rewrite the differential inequality as
\begin{equation}\label{eq:eprime-Q}
\frac{d}{dt}e(t)
 \le -\,(1-Ce^{1/2})\,Q[f] + C e^{3/2}.
\end{equation}
Combining Corollary~\ref{cor:gap-kminus1} with \eqref{eq:eprime-Q}, using also $e\le1$, yields
\begin{align*}
\frac{d}{dt}e
\le -(1-Ce^{1/2})(\widehat\lambda_\omega e - 4\omega\pi\widehat\lambda_\omega e^2) + C e^{3/2}
\le -\widehat\lambda_\omega\,e(t) + C' e(t)^{3/2},
\end{align*}
where $C'=C'(\omega)$.

We now apply Lemma~\ref{lem:ODE-comparison} with
\(
\alpha=\widehat\lambda_\omega
=\frac74+\delta_\omega\), \(\sigma=\frac12
\)
to conclude the result.
\end{proof}

\section{Decay of the derivative of curvature and barycentre}

The constant $\varepsilon_\omega$ in \eqref{eq:koscsmall} is set to be that of Corollary \ref{cor:ODE}.
With exponential decay of $\Kosc$ for the rescaled flow now established, the argument from here to smooth convergence could be carried out `from scratch', an approach commonly taken in the literature. 
We opt instead to present another new integral estimate, which has two benefits.
First, it shortens the overall convergence proof, enabling us to apply \cite[Theorem 1.2]{MW25}.
Second, it may be of independent interest, showing that decay of the curvature in $L^2$ implies decay of the derivative of curvature in $L^2$; a (new) regularity property of the rescaled free elastic flow.

\begin{lemma}
Let $\gamma:\S^1\times[0,\infty)\to\R^2$ be a rescaled free elastic flow \eqref{eq:rescaled}.
We have
\begin{align}
\frac{d}{dt}||k_s||_{L^2}^2
&=
 - 4\int k_{sss}^2\,\ds
    + 10\int k_{ss}^2k^2\,\ds
    - \frac{10}{3}\int k_s^4\,\ds
    - 11\int k_s^2k^4\,\ds
\notag\\&\qquad
- \frac{3}{2\omega\pi}\int k_s^2\,\ds \left(
    2\int k_s^2\,\ds - \int k^4\,\ds
 \right)
 \,.
\label{EQevolks}
\end{align}
\end{lemma}
\begin{proof}
First, the evolution equation for the derivative of curvature is
\[
\frac{d}{dt}k_{s} = 
  (F+\lambda(t)(\gamma\cdot\nu))_{sss} + (F+\lambda(t)(\gamma\cdot\nu))_sk^2 + 3(F+\lambda(t)(\gamma\cdot\nu))k_sk
\,.
\]
Thus 
\begin{align*}
\frac{d}{dt}(k_s^2\,\ds)
&= \left( 2k_s(F+\lambda(t)(\gamma\cdot\nu))_{sss} 
+ 2(F+\lambda(t)(\gamma\cdot\nu))_sk_sk^2 
+ 5(F+\lambda(t)(\gamma\cdot\nu))k_s^2k
\right)\,\ds
\,.
\end{align*}
Therefore
\begin{align*}
\frac{d}{dt}||k_s||_{L^2}^2
 &= 2\int
     F_{sss}k_s
    \,\ds
    +2\int
     F_sk_sk^2
    \,\ds
    +5\int
     Fk_s^2k
    \,\ds
\\
  &+\lambda(t)\left(2\int
     (\gamma\cdot\nu)_sk_{sss}
    \,\ds
    +2\int
     (\gamma\cdot\nu)_sk_sk^2
    \,\ds
    +5\int
     (\gamma\cdot\nu)k_s^2k
    \,\ds
    \right)
    \,.
\end{align*}

Let us now work on the terms multiplying $\lambda(t)$.
To prepare, we calculate some derivatives of $\gamma\cdot\nu$:
\begin{align*}
(\gamma\cdot\nu)_{sss}
 &= 
(-k\gamma\cdot\tau)_{ss}
\\ &= 
(-k_s\gamma\cdot\tau-k-k^2\gamma\cdot\nu)_{s}
\\ &= 
 -k_{ss}\gamma\cdot\tau
 -k_s
 -kk_s\gamma\cdot\nu
 -k_s
 -2kk_s\gamma\cdot\nu
 +k^3\gamma\cdot\tau
 \,.
\end{align*}
Integrating by parts then reveals
\begin{align*}
    2&\int
     (\gamma\cdot\nu)_sk_{sss}
    \,\ds
    +2\int
     (\gamma\cdot\nu)_sk_sk^2
    \,\ds
    +5\int
     (\gamma\cdot\nu)k_s^2k
    \,\ds
\\
&=
    2\int
      -k_{ss}k_s\gamma\cdot\tau
     +k^3k_s\gamma\cdot\tau
 -2k_s^2
 -3kk_s^2\gamma\cdot\nu
    \,\ds
\\&\quad
    -2\int
     \gamma\cdot\tau k_sk^3
    \,\ds
    +5\int
     (\gamma\cdot\nu)k_s^2k
    \,\ds
\\
&=
    \int
      -2k_{ss}k_s\gamma\cdot\tau
 -4k_s^2
 -kk_s^2\gamma\cdot\nu
    \,\ds
\\
&=
    \int
      k_s^2k\gamma\cdot\nu
      + k_s^2
 -4k_s^2
 -kk_s^2\gamma\cdot\nu
    \,\ds
= -3\int k_s^2\,\ds
\,.
\end{align*}
The result thus follows by simplifying the $F$ terms exactly as in \cite[Lemma 3.1]{MW25} (note that here our $F$ is $-2$ times the $F$ in \cite{MW25}).
\end{proof}

\begin{proposition}
\label{PPdecayks}
Let $\gamma:\S^1\times[0,\infty)\to\R^2$ be a rescaled free elastic flow \eqref{eq:rescaled} satisfying \eqref{eq:koscsmall}.
There exists a constant $t_\omega  \ge0$ depending only on $\omega$ such that 
\begin{equation}
\|k_s\|_{L^2}^2(t)
\le \|k_s\|_{L^2}^2(t_\omega)e^{-\frac1{2\omega^6} (t-t_\omega)}\,,\qquad t\ge t_\omega.
\label{eq:ksexpdecay}
\end{equation}
\end{proposition}
\begin{proof}
Set
\[
H(t):=\int k_s^2k^4\,\ds- \frac1{2\omega\pi}\left(\int k_s^2\,\ds\right) \left(\int k^4\,\ds\right).
\]
Using \eqref{EQevolks} and \ra{the explicit estimate
\begin{align*}
10\int k_{ss}^2k^2\,\ds
&= -10\int k_{sss}k_sk^2\,\ds
   +\frac{20}{3}\int k_s^4\,\ds  \\
&\le \delta_0\|k_{sss}\|_{L^2}^2
   +\frac{25}{\delta_0}\int k_s^2k^4\,\ds
   +\frac{20}{3}\int k_s^4\,\ds,
\end{align*}}
together with the identity
\[
-11\int k_s^2k^4\,\ds- \frac{3}{2\omega\pi}\int k_s^2\,\ds\left(2\int k_s^2\,\ds-\int k^4\,\ds\right)
= -8\int k_s^2k^4\,\ds-3H(t)-\frac{3}{\omega\pi}\left(\int k_s^2\,\ds\right)^2,
\]
we obtain
\begin{align}
\label{eqn:ksderiv}
\frac{d}{dt}\|k_s\|_{L^2}^2
&\le -(4-\delta_0)\|k_{sss}\|_{L^2}^2
-\left(8-\frac{25}{\delta_0}\right)\int k_s^2k^4\,\ds
\\
\notag
&\qquad\qquad+\frac{10}{3}\|k_s\|_{L^4}^4
-3H(t)- \frac{3}{\omega\pi}\left(\int k_s^2\,\ds\right)^2.
\end{align}
\ra{Here the coefficient \(25/\delta_0\) is exactly the Young constant in the preceding displayed estimate, and the coefficient \(10/3\) is obtained by combining \(-10/3\int k_s^4\,\mathrm ds\) from \eqref{EQevolks} with \(20/3\int k_s^4\,\mathrm ds\).}
To estimate $H(t)$, first write
\[
H(t)=\int k_s^2\big(k^4-\overline{k^4}\big)\,\ds
=\int k_s^2\Big(4f+6\big(f^2-\overline{f^2}\big)
+4\big(f^3-\overline{f^3}\big)+\big(f^4-\overline{f^4}\big)\Big)\,\ds.
\]
For the linear term, using
\eqref{eq:gns}  with $j=1$, $m=3$, $p=q=r=2$ and $\theta=\frac13$
%
%
yields 
\begin{equation}
\label{eqn:gns1}
\|k_s\|_{L^2}^3 \le \|f\|_{L^2}^{2}\|k_{sss}\|_{L^2}.
\end{equation}
Poincar\'e implies $\|f\|_{L^2}\le \omega^{3} \|k_{sss}\|_{L^2}$,
thus $\|k_s\|_{L^2}^3 \le \omega^{3}\sqrt{e(t)}\|k_{sss}\|_{L^2}^2$ and 
\begin{align*}
\int k_s^2|f|\,\ds
&\le \|f\|_{L^\infty}\|k_s\|_{L^2}^2
\le \|k_s\|_{L^1}\|k_s\|_{L^2}^2
\le (2\omega\pi)^{1/2}\|k_s\|_{L^2}^3
\le C\omega^{\frac72}\sqrt{e(t)}\|k_{sss}\|_{L^2}^2\,.
\end{align*}
We also estimate, using H\"older, the fundamental theorem of calculus, and the Poincar\'e inequality
\[
\int k_s^2\Big|f^2-\frac{e(t)}{2\omega\pi}\Big|\,\ds
\le \int k_s^2 f^2\,\ds+\frac{e(t)}{2\omega\pi}\int k_s^2\,\ds
\le e(t)(||k_s||_{L^\infty}^2 + C\omega^{-1}||k_s||_{L^2}^2)
\le C\omega^3 e(t) \|k_{sss}\|_{L^2}^2.
\]
For the cubic and quartic part of $H(t)$, if $P(x):=4x^3+x^4$, the fundamental theorem of calculus, H\"older and Young gives
\begin{align*}
|P(f)-\overline{P(f)}|
&\le \int |(P(f))_s|\,\ds
\le 12\int |k_s|f^2\,\ds + 4\int |k_s||f|^3\,\ds
\le ||k_s||_{L^1}(12||f||_{L^\infty}^2
 + 4||f||_{L^\infty}^3)
\\
&\le ||k_s||_{L^1}^3(12
 + 4||k_s||_{L^1})
\le \delta + C_{\delta}\,\omega^2\|k_s\|_{L^2}^4,
\end{align*}
and therefore, using again Poincar\'e 
and \eqref{eqn:gns1}:
\[
\int k_s^2\big|4(f^3-\overline{f^3})+(f^4-\overline{f^4})\big|\,\ds
\le \delta\,\|k_{s}\|_{L^2}^2 + C_\delta\,\omega^2\|k_s\|_{L^2}^6
\le \delta\,\omega^4\|k_{sss}\|_{L^2}^2 + C_\delta\,\omega^2\,e(t)^2\|k_{sss}\|_{L^2}^2.
\]
Combining the preceding bounds, we obtain the requisite estimate for $H(t)$:
\begin{equation}
\label{eqn:hest}
-3H(t)\le 3|H(t)|
\le C\Big(\omega^{\frac72}\sqrt{e(t)}
    + \omega^3e(t)
    + \delta\omega^4
    + C_\delta\omega^2\,e(t)^2\Big)\|k_{sss}\|_{L^2}^2\,,\qquad
0<\delta\le1.
\end{equation}
The only term remaining to estimate in \eqref{eqn:ksderiv} is $(10/3)||k_s||_{L^4}^4$.
For this it suffices to apply \eqref{eq:gns}  with $j=1$, $m=3$, $p=4$, $q=r=2$ and $\theta=\frac5{12}$, which is the estimate
%
%
\begin{equation*}
\|k_s\|_{L^4}^4
  \le C\|f\|_{L^2}^{7/3}\|k_{sss}\|_{L^2}^{5/3}
.
\end{equation*}
Using Poincaré, \(\|f\|_{L^2}\le \omega^3\|k_{sss}\|_{L^2}\), we obtain
\begin{equation}
\label{eqn:l4est}
  \|k_s\|_{L^4}^4
  \le C\omega\, e(t)\,\|k_{sss}\|_{L^2}^2 .
\end{equation}
Now, combining the estimates \eqref{eqn:hest}, \eqref{eqn:l4est} with the evolution equation \eqref{eqn:ksderiv}, discarding good terms, and choosing $\delta_0 = 25/8$, we arrive at
\begin{align*}
\frac{d}{dt}\|k_s\|_{L^2}^2
&\le -\Big(\frac78-C\omega^{\frac72}\sqrt{e(t)}
    - C\omega^3e(t)
    - C\delta\omega^4
    - C_\delta\omega^2\,e(t)^2
    - C\omega e(t)\Big)\|k_{sss}\|_{L^2}^2
\\&=: -\Big(\frac78-F(\omega,e(t),\delta)\Big)\|k_{sss}\|_{L^2}^2.
\end{align*}
Observe that $F(\omega,e(t),1/(8C\omega^4))\to\frac18$ (recall that Corollary~\ref{cor:ODE} implies $e(t)\to0$ exponentially), and so take $t\ge t_\omega$ where $F(\omega,e(t),1/(8C\omega^4))\le\frac38$ for all $t\ge t_\omega$.
Observe that $t_\omega$ depends only upon $\omega$ and the estimate \eqref{eq:ode-final-improved}, which depends again upon $\omega$ and additionally an upper bound on $e(0)$. However as we are working under the assumption $e(0)\le\varepsilon_\omega$, and $\varepsilon_\omega$ depends only on $\omega$, we obtain that $t_\omega$ also finally depends only on $\omega$.

Then, using also Poincar\'e, 
\[
\frac{d}{dt}\|k_s\|_{L^2}^2
\le -\frac12\|k_{sss}\|_{L^2}^2
\le -\frac12\omega^{-6}\|k_s\|_{L^2}^2.
\]
The stated exponential decay follows from Gr\"onwall.
\end{proof}

Our convergence statement uses the following invariant distance.

\begin{definition}[Invariant $W^{2,2}$ distance]\label{def:Sob-GH}
For two $W^{2,2}$ arclength parametrised immersions
$\gamma_1,\gamma_2:\S^1\to\R^2$ define
\begin{equation}\label{eq:d2-def}
  \mathbf d_2(\gamma_1,\gamma_2)
  :=\inf_{\sigma\in\S^1}
      \|\gamma_1-\gamma_2(\cdot+\sigma)\|_{W^{2,2}(\S^1)} .
\end{equation}
\end{definition}

The next proposition is used only after Proposition~\ref{PPdecayks} and
\cite[Theorem~1.2]{MW25} have supplied smooth convergence to a unit
\(\omega\)-circle.

\begin{proposition}[Barycentre decay]\label{prop:barycentre-decay}
Let $\gamma:\S^1\times[0,\infty)\to\R^2$ be a rescaled free elastic flow \eqref{eq:rescaled} satisfying \eqref{eq:koscsmall} 
that converges smoothly to a unit $\omega$-circle.

Then the barycentre of the limiting circle is the origin and, for some constant $C=C(\gamma(\cdot,0),\omega)$,
\begin{equation}\label{eq:cm-decay}
  \left|\overline\gamma(t)\right|
  \le C e^{-t}\qquad (t\ge0).
\end{equation}
\end{proposition}
\begin{proof}
Use the normal angle $\vartheta$, in terms of which we express the outward normal $n$ and tangent $\tau$ as
\[
 n=(\cos\vartheta,\sin\vartheta),\qquad
 \tau=(-\sin\vartheta,\cos\vartheta)\,.
\]
We further have the following formulae for the support function $h$ and radius of curvature $\rho$:
\[
 h(\vartheta,t)=\gamma(\vartheta,t)\cdot n(\vartheta)
 =-\gamma(\vartheta,t)\cdot\nu(\vartheta,t),
 \qquad
 \rho(\vartheta,t)=k^{-1}=h+h_{\vartheta\vartheta}.
\]
Set $u:=\rho-1$ and $g:=h-1$.

Apply Corollary \ref{cor:ODE}, so that \eqref{eq:ode-final-improved} holds (and for $t\ge t_\omega$, \eqref{eq:ksexpdecay} holds as well).
Smooth convergence implies that for all $t$ sufficiently large, $\rho$ is in a neighbourhood of 1. 
Fix \(t_*=t_*(\gamma(\cdot,0),\omega)\) sufficiently large that
\(1/2\le\rho\le2\) and all  estimates used below hold for
\(t\ge t_*\).

Then, because \(
  \frac12\int u^2\,d\vartheta\le \int \frac{(1-\rho)^2}{\rho}\,d\vartheta = e(t) \le 2\int u^2\,d\vartheta
\), we also have
\begin{equation}\label{eq:e-decay}
  \|u(\cdot,t)\|_{L^2(d\vartheta)}
  \le C e(t)^{1/2}
  \le C e^{-\frac12\widehat\lambda_\omega t},
  \qquad t\ge t_* .
\end{equation}
In the estimates below, \(C\) may vary from line to line.  Constants using only
the small-energy estimates depend on \(\omega\); constants using the
assumed smooth convergence also depend on the initial datum $\gamma(\cdot,0)$.

Decompose
\[
  g(\vartheta,t)=a_1(t)\cos\vartheta+a_2(t)\sin\vartheta+w(\vartheta,t),
\]
with $w$ orthogonal to $1,\cos\vartheta,\sin\vartheta$.  Since
$u=(I+\partial_\vartheta^2)g=(I+\partial_\vartheta^2)w$, we may estimate using a simple argument
\begin{equation}\label{eq:w-invert}
  \|w\|_{H^2(d\vartheta)}\le C\|u\|_{L^2(d\vartheta)} .
\end{equation}
Smooth convergence gives, for each \(m\), a finite constant
\(C_{\omega,m,\gamma_0}\) such that
\(
  \sup_{t\ge t_*}\|u(\cdot,t)\|_{H^m(d\vartheta)}^2
  \le C_{\omega,m,\gamma_0}
\).
Using this for $m=6$, we estimate
\begin{equation}\label{eq:u-C2-decay}
\|u\|_{C^2}
\le C\|u\|_{H^3}
\le C\|u\|_{L^2}^{1/2}\|u\|_{H^6}^{1/2}
\le C e(t)^{1/4}.
\end{equation}
Let $\phi_i\in\{\cos\vartheta,\sin\vartheta\}$ and
$a_i=(\omega\pi)^{-1}\int h\phi_i\,d\vartheta$.
Using \eqref{eq:rescaled}
\[
  a_i'=\frac1{\omega\pi}\int (2k_{ss}+k^3+\lambda h)\phi_i\,d\vartheta.
\]
We write
\(
  2k_{ss}+k^3+\lambda h
  =(2k_{ss}+k^3-1)-g+(\lambda+1)h
\). 
The projection of $-g$ is $-a_i$.  For the remaining geometric part, using
$k=(1+u)^{-1}$ and $\partial_s=(1+u)^{-1}\partial_\vartheta$, a direct calculation yields
\[
2k_{ss}+k^3-1=-(2u_{\vartheta\vartheta}+3u)+R(u),
\]
where, for $\|u\|_{C^0}$ small,
\(
  |R(u)|\le C\big(|u|\,|u_{\vartheta\vartheta}|+u_\vartheta^2+u^2\big)
\). 
Since $u$ has no first harmonics,
\[
  \int (2u_{\vartheta\vartheta}+3u)\phi_i\,d\vartheta
  =\int u(2\phi_i''+3\phi_i)\,d\vartheta
  =\int u\phi_i\,d\vartheta=0.
\]
Thus, using \eqref{eq:u-C2-decay} and
$\int u_\vartheta^2\,d\vartheta=-\int u u_{\vartheta\vartheta}\,d\vartheta$,
\begin{equation}\label{eq:geom-first-mode}
  \left|\int(2k_{ss}+k^3-1)\phi_i\,d\vartheta\right|
  \le C\int |R(u)|\,d\vartheta
  \le C\left( \|u_{\vartheta\vartheta}\|_{L^\infty} \|u\|_{L^2} + e(t)\right)
  \le C e(t)^{3/4}\,.
\end{equation}
Equivalently, keeping the translation mode in its exact linear form, we have
\begin{equation}\label{eq:a-lambda-form}
  a'(t)=\lambda(t)a(t)+\mathcal G(t),
  \qquad
  |\mathcal G(t)|\le C e(t)^{3/4}.
\end{equation}
Indeed, for \(\phi_i\in\{\cos\vartheta,\sin\vartheta\}\),
\[
  \mathcal G_i(t)
  :=
  \frac1{\omega\pi}\int (2k_{ss}+k^3-1)\phi_i\,d\vartheta,
\]
and the estimate is precisely \eqref{eq:geom-first-mode}.

Recall that (see \eqref{eq:rescaled})
\(
  \lambda(t)+1=\frac1{2\omega\pi}\left(2\int k_s^2\,ds-
  \int(k^4-1)\,ds\right)
\). 
Now 
$\int k_s^2\,ds=\int(1+u)^{-5}u_\vartheta^2\,d\vartheta
\le C e(t)^{3/4}$, while the integrand of the second term 
$\int(k^4-1)\,ds=\int((1+u)^{-3}-(1+u))\,d\vartheta$ has vanishing linear part because
$\int u\,d\vartheta=0$.
Hence
\begin{equation}\label{eq:lambda-plus-one}
  |\lambda(t)+1|\le C e(t)^{3/4}.
\end{equation}
Since the amplitudes $a_i$ are bounded for $t\ge t_*$, \eqref{eq:geom-first-mode}
and \eqref{eq:lambda-plus-one} give
\begin{equation}\label{eq:a-ODE-clean}
  a_i'(t)=-a_i(t)+\mathcal N_i(t),
  \qquad |\mathcal N_i(t)|\le C e(t)^{3/4}\qquad(t\ge t_*).
\end{equation}
Because $\frac34\widehat\lambda_\omega>1$, Lemma~\ref{lem:duhamel} implies
\begin{equation}\label{eq:a-decay}
  |a(t)|\le C e^{-t}\qquad(t\ge t_*).
\end{equation}
In particular
$a(t)\to0$, so the limiting circle is centred at the origin.
Indeed, if the limiting unit circle has centre \(p\), its support function is
\(1+p\cdot n\), so \(a(t)\to p\).

Finally, (recalling $\gamma = hn+h_\vartheta\tau$, $h=1+g$, and integrating by parts)
\begin{align}
  \overline\gamma(t) &= \frac1{2\omega\pi}\int
  (h n+h_\vartheta\tau)\rho\,d\vartheta 
  = \frac1{2\omega\pi}\int
  (h^2n 
  - h_{\vartheta\vartheta}h_\vartheta\tau)\,d\vartheta 
  \notag
  \\
  &= \frac1{2\omega\pi}\int
  ((1+g)^2n 
  - g_{\vartheta\vartheta}g_\vartheta\tau)\,d\vartheta 
  = \frac1{\omega\pi} \int gn\,d\vartheta
  + \frac1{2\omega\pi}\int
  (g^2n 
  - g_{\vartheta\vartheta}g_\vartheta\tau)\,d\vartheta 
  .
\label{eq:gammabar}
\end{align}
The first term (the linear part) is exactly $a(t)$. 
The remaining terms are quadratic, and we claim
\begin{equation}
\label{eq:claim}
  |\overline\gamma(t)-a(t)|
  \le C\big(|a(t)|\|u\|_{L^2(d\vartheta)}+\|u\|_{L^2(d\vartheta)}^2\big).
\end{equation}
To see this, recall $g = b + w$ where $b=a\cdot n$, $a = (a_1,a_2)$.
Since
\(
  b n+b_\vartheta\tau
  =(a\cdot n)n+(a\cdot\tau)\tau
  =a
\)
we have
\[
  \int \big(b^2n-b_{\vartheta\vartheta}b_\vartheta\tau\big)\,d\vartheta
  =
  \int b\big(bn+b_\vartheta\tau\big)\,d\vartheta
  =
  a\int b\,d\vartheta
  =0 .
\]
Combining with \eqref{eq:gammabar} we find
\begin{align*}
  |\overline\gamma(t)-a(t)|
&= \bigg|\frac1{2\omega\pi}\int (b^2n - b_{\vartheta\vartheta}b_\vartheta\tau)\,d\vartheta
+ \frac1{2\omega\pi}\int (w^2n - w_{\vartheta\vartheta}w_\vartheta\tau)\,d\vartheta
\\&\qquad\qquad\qquad\qquad + \frac1{2\omega\pi}\int (2bw\,n - (
b_{\vartheta\vartheta}w_\vartheta
+ w_{\vartheta\vartheta}b_\vartheta
)\tau)\,d\vartheta
\bigg|
\\&= \bigg|\frac1{2\omega\pi}\int (w^2n - w_{\vartheta\vartheta}w_\vartheta\tau)\,d\vartheta
+ \frac1{2\omega\pi}\int (2bw\,n - (
b_{\vartheta\vartheta}w_\vartheta
+ w_{\vartheta\vartheta}b_\vartheta
)\tau)\,d\vartheta
\bigg|
\\&\le C
  \int
  \Big(
    |b||w|
    + |b_{\vartheta\vartheta}||w_\vartheta|
    + |b_\vartheta||w_{\vartheta\vartheta}|
    + |w|^2
    + |w_\vartheta||w_{\vartheta\vartheta}|
  \Big)\,d\vartheta                                      \\
  &\le C\Big(
    |a(t)|\|w\|_{H^{2}(d\vartheta)}
    + \|w\|_{H^{2}(d\vartheta)}^2
  \Big)                                                   
  \le C\Big(
    |a(t)|\|u\|_{L^2(d\vartheta)}
    + \|u\|_{L^2(d\vartheta)}^2
  \Big);
\end{align*}
where we used \eqref{eq:w-invert}.
Thus we have shown \eqref{eq:claim}.
Together with \eqref{eq:e-decay} and \eqref{eq:a-decay}, this proves
\eqref{eq:cm-decay} for \(t\ge t_*\). Enlarging \(C\) on the finite interval
\([0,t_*]\) gives \eqref{eq:cm-decay} for all \(t\ge0\).
\end{proof}

\section{Proof of the main results}

\begin{proof}[Proof of Theorem \ref{thm:main}]
Corollary~\ref{cor:ODE} implies the stated decay of $\Kosc$. 
Proposition~\ref{PPdecayks} gives \(\|k_s\|_{L^2}^2(t)\to0\). Hence, for some
\(t_1\), the smallness hypothesis of \cite[Theorem~1.2]{MW25} is satisfied
for the flow restarted at \(t_1\) (equivalently, for the corresponding
unrescaled curve at time \(\tGamma(t_1)\)). The time change between the
unrescaled and rescaled evolutions is monotone, so \cite[Theorem~1.2]{MW25}
gives smooth convergence of the rescaled flow to a unit \(\omega\)-circle.
Proposition~\ref{prop:barycentre-decay}
then gives \eqref{eq:cm-decay}, so this limiting circle is the one centred at
the origin, denoted $\gamma_\omega$.

It remains to show \eqref{eq:d2conv}.
Put both curves in the
arclength gauge on $\S^1$ and choose the shift of $\gamma_\omega$ so that
the angle difference $\varphi$ between the two tangents has zero mean.  Then
$\varphi_s=k-1$, and Poincar\'e gives
\[
  \|\tau-\tau_\omega\|_{L^2}+\|\nu-\nu_\omega\|_{L^2}
  \le C_\omega\|\varphi\|_{L^2}
  \le C_\omega e(t)^{1/2}.
\]
Consequently
\[
  \|(\gamma-\gamma_\omega)_s\|_{L^2}
  +\|(\gamma-\gamma_\omega)_{ss}\|_{L^2}
  \le C_\omega e(t)^{1/2},
\]
because $(\gamma-\gamma_\omega)_{ss}=(k-1)\nu+(\nu-\nu_\omega)$.  Since
$\int_{\gamma_\omega}\gamma_\omega\,ds=0$, Poincar\'e applied to
$\gamma-\gamma_\omega-\overline\gamma(t)$ yields
\begin{equation}\label{eq:W22-sharp}
  \mathbf d_2(\gamma(\cdot,t),\gamma_\omega)
  \le C_\omega e(t)^{1/2}+C_\omega|\overline\gamma(t)| .
\end{equation}
Using \eqref{eq:cm-decay} and the decay of $e(t)$,
\[
  \mathbf d_2(\gamma(\cdot,t),\gamma_\omega)
  \le C\left(e^{-\frac12\widehat\lambda_\omega t}+e^{-t}\right)
  \le C e^{-\min\{\widehat\lambda_\omega/2,1\}t}.
\]
Finally,
\[
  \min\{\widehat\lambda_\omega/2,1\}
  =
  \frac12\left(\frac74+\min\left\{\frac14,\delta_\omega\right\}\right),
\]
which is exactly \eqref{eq:d2conv} with $r_\omega=\min\{\frac14,\delta_\omega\}$.
\end{proof}

\subsection{A pre-eminent choice of origin}\label{subsec:preeminent-origin}

We now remove the translation mode by choosing the origin at the limiting
centre of mass of the unrescaled flow.

\begin{proof}[Proof of Theorem~\ref{thm:preeminent-origin}]
Let \(\Gamma(\cdot,\tGamma)\) be the unrescaled flow and let
\(\gamma(\cdot,t)\) be its length-normalised rescaling.  We write
\(\sigma(t):=\sigma(\tGamma(t))\), so that
\(
  \Gamma(\cdot,\tGamma(t))=\sigma(t)\gamma(\cdot,t)\),
  \(
  \frac{d\tGamma}{dt}=\sigma(t)^4
\), and (as in the derivation of \eqref{eq:rescaled}),
\( \sigma_t=-\lambda(t)\sigma\).
Set
\(
  \alpha_\omega:=\frac34\widehat\lambda_\omega>1\) (as noted between \eqref{eq:a-ODE-clean} and \eqref{eq:a-decay}).

We use the notation from Proposition~\ref{prop:barycentre-decay}. Thus
\(h=1+g\), \(u=\rho-1\), \(a(t)=(a_1(t),a_2(t))\), and
\(
  g(\vartheta,t)=a(t)\cdot n(\vartheta)+w(\vartheta,t).
\)
Since 
\(
  \frac{d}{dt}\big(\log\sigma(t)-t\big)
  =-(\lambda(t)+1)
\), 
the estimate \eqref{eq:lambda-plus-one} implies $(\lambda(t)+1)$ is integrable, hence
\begin{equation}\label{eq:sigma-asymp-preem}
  \sigma(t)=c_*e^t(1+o(1)),
  \qquad c_*>0.
\end{equation}
In particular, \(\sigma(t)\simeq e^t\) for large \(t\).

Define the unrescaled translation mode
\[
  B(t):=\sigma(t)a(t).
\]
Differentiating reveals
\(
  B'(t)
  =
  \sigma_t a+\sigma a'
  =
  -\lambda\sigma a+\sigma(\lambda a+\mathcal G)
  =
  \sigma(t)\mathcal G(t).
\)
Therefore, by \eqref{eq:e-decay}, \eqref{eq:a-lambda-form}, and \eqref{eq:sigma-asymp-preem},
\[
  |B'(t)|
  \le C e^{-(\alpha_\omega-1)t}.
\]
Since \(\alpha_\omega>1\), \(B'\) is integrable on a tail, and hence \(B(t)\)
converges. 
Set 
\(  p_\infty:=\lim_{t\to\infty}B(t)
  =
  \lim_{t\to\infty}\sigma(t)a(t)
\), and note that
\(
  |B(t)-p_\infty|
  \le C e^{-(\alpha_\omega-1)t}
\). 
Let us now identify \(p_\infty\) as the limiting centre of mass of the unrescaled
flow: $p_\infty$ is our pre-eminent choice of origin.  
Since \(\Gamma=\sigma\gamma\) and \(ds_\Gamma=\sigma ds\), we have \(  \overline\Gamma(\tGamma(t))=\sigma(t)\overline\gamma(t)
\) and the geometric estimate \eqref{eq:claim} gives
\[
  |\overline\gamma(t)-a(t)|
  \le
  C\big(
    |a(t)|\|u\|_{L^2(d\vartheta)}
    +\|u\|_{L^2(d\vartheta)}^2
  \big).
\]
Multiplying by \(\sigma(t)\), and using \(B=\sigma a\), we find
\[
  \sigma(t)|\overline\gamma(t)-a(t)|
  \le
  C\Big(
    |B(t)|\|u\|_{L^2(d\vartheta)}
    +\sigma(t)\|u\|_{L^2(d\vartheta)}^2
  \Big).
\]
The first term tends to zero because \(B(t)\) converges and
\(\|u\|_{L^2}\to0\).  The second tends to zero by
\eqref{eq:e-decay} and \(\sigma(t)\simeq e^t\), since
\(\widehat\lambda_\omega>1\).  Hence
\[
  \lim_{\tGamma\to\infty}\overline\Gamma(\tGamma)
  =
  \lim_{t\to\infty}\sigma(t)\overline\gamma(t)
  =
  \lim_{t\to\infty}\sigma(t)a(t)
  =
  p_\infty.
\]

Now translate the unrescaled flow by $p_\infty$:
\[
  \Gamma^\sharp(\cdot,\tGamma)
  :=
  \Gamma(\cdot,\tGamma)-p_\infty,
  \qquad
  \gamma^\sharp(\cdot,t)
  :=
  \sigma(t)^{-1}\Gamma^\sharp(\cdot,\tGamma(t))
  =
  \gamma(\cdot,t)-\frac{p_\infty}{\sigma(t)}.
\]
All geometric quantities remain unchanged, including \(e(t)\).
Translation by
\(q\in\mathbb R^2\) sends the support function \(h\) to \(h-q\cdot n\), so
the first-harmonic vector of \(\gamma^\sharp\) is
\[
  a^\sharp(t)
  =
  a(t)-\frac{p_\infty}{\sigma(t)}
  =
  \frac{B(t)-p_\infty}{\sigma(t)},
\]
which (using \(
  |B(t)-p_\infty|
  \le C e^{-(\alpha_\omega-1)t}
\) and \(\sigma(t)\simeq e^t\)) obeys the estimate 
\(
  |a^\sharp(t)|
  \le C e^{-\alpha_\omega t}
\). 
Applying \eqref{eq:claim} to \(\gamma^\sharp\), with \(a^\sharp\) in place of
\(a\), gives
\[
  |\overline{\gamma^\sharp}(t)-a^\sharp(t)|
  \le
  C\big(
    |a^\sharp(t)|\|u\|_{L^2(d\vartheta)}
    +\|u\|_{L^2(d\vartheta)}^2
  \big),
\]
which implies \(
  |\overline{\gamma^\sharp}(t)|
  \le C e^{-\alpha_\omega t}
\). 
Thus the estimate \eqref{eq:W22-sharp} applied to
\(\gamma^\sharp\) yields
\begin{equation}
\label{eq:gammasharpconv}   
  \mathbf d_2\big(\gamma^\sharp(\cdot,t),\gamma_\omega\big)
  \le
  C e^{-\frac12\widehat\lambda_\omega t}
  + C|\overline{\gamma^\sharp}(t)|
  \le
  C e^{-\frac12\widehat\lambda_\omega t}.
\end{equation}
After enlarging \(C\), this holds on the initial finite interval as well.

It remains to translate this estimate into unrescaled time.  From
\eqref{eq:sigma-asymp-preem} and \(d\tGamma/dt=\sigma(t)^4\),
\[
  \tGamma(t)
  =
  \int_0^t \sigma(s)^4\,ds
  =
  \frac{c_*^4}{4}e^{4t}(1+o(1)).
\]
Thus, for all sufficiently large \(\tGamma\),
\[
  e^{-\frac12\widehat\lambda_\omega t(\tGamma)}
  \le
  C(1+\tGamma)^{-\frac18\widehat\lambda_\omega}.
\]
Combining this with \eqref{eq:gammasharpconv}, we get
\[
  \mathbf d_2\big(\gamma^\sharp(\cdot,t(\tGamma)),\gamma_\omega\big)
  \le
  C(1+\tGamma)^{-\,\frac18\widehat\lambda_\omega}
  =
  C(1+\tGamma)^{-\,\frac18\left(\frac74+\delta_\omega\right)}
\]
for all \(\tGamma\ge\tGamma_0\), as claimed.
\end{proof}
\appendix

\section*{Appendix}

\setcounter{equation}{0}
\renewcommand{\theequation}{A.\arabic{equation}}
\renewcommand{\theHequation}{A.\arabic{equation}}
\begin{lemma}[Variation-of-constants/Duhamel bound]\label{lem:duhamel}
{Let $t_0\ge0$} and let $a:[t_0,\infty)\to\R$ solve
\[
a'(t)=-\,a(t)+F(t)\qquad (t\ge t_0),
\]
with $|F(t)|\le C_F\,\mathrm e^{-\alpha t}$ for some $\alpha>0$. Then, for all $t\ge t_0$,
\begin{equation}\label{eq:duhamel-general}
|a(t)|\ \le\ \mathrm e^{-(t-t_0)}|a(t_0)|\ +\ 
\begin{cases}
\displaystyle \frac{C_F}{|1-\alpha|}\left|\mathrm e^{-\alpha t}-\mathrm e^{-(t-t_0)}\mathrm e^{-\alpha t_0}\right|, & \alpha\neq 1,\\[0.6em]
\displaystyle C_F\,(t-t_0)\,\mathrm e^{-t}, & \alpha=1.
\end{cases}
\end{equation}
In particular, for a constant \(C\) depending on \(a(t_0)\), \(C_F\),
\(t_0\), and \(\alpha\),
\[
|a(t)|\le
\begin{cases}
C e^{-\alpha t}, & 0<\alpha<1,\\
C(1+t)e^{-t}, & \alpha=1,\\
C e^{-t}, & \alpha>1.
\end{cases}
\]
\end{lemma}
\begin{proof}
Multiply the ODE by $\mathrm e^{t}$ to get $(\mathrm e^{t}a(t))'=\mathrm e^{t}F(t)$ and integrate:
\[
a(t)=\mathrm e^{-(t-t_0)}a(t_0)+\int_{t_0}^{t}\mathrm e^{-(t-s)}F(s)\,ds.
\]
The bounds in \eqref{eq:duhamel-general} follow by estimating the integral explicitly:
for $\alpha\neq1$,
\[
\int_{t_0}^{t}\mathrm e^{-(t-s)}\mathrm e^{-\alpha s}\,ds
=\mathrm e^{-t}\int_{t_0}^{t}\mathrm e^{(1-\alpha)s}\,ds
=\frac{1}{1-\alpha}\Big(\mathrm e^{-\alpha t}-\mathrm e^{-(t-t_0)}\mathrm e^{-\alpha t_0}\Big),
\]
and hence the stated bounds.
When \(\alpha=1\), the integral equals \((t-t_0)\mathrm e^{-t}\).
The three stated consequences follow immediately from
\eqref{eq:duhamel-general}: if \(0<\alpha<1\), both terms are bounded by
\(C\mathrm e^{-\alpha t}\); if \(\alpha>1\), both terms are bounded by
\(C\mathrm e^{-t}\); and if \(\alpha=1\), they are bounded by
\(C(1+t)\mathrm e^{-t}\). This proves the lemma.
\end{proof}

\bibliographystyle{plain}
\bibliography{FEF}

\end{document}